\documentclass[12pt,twoside]{article}
\usepackage{amssymb,amsmath,amsthm,latexsym}
\usepackage{amsfonts}
\usepackage{graphicx}
\newtheorem{thm}{Theorem}[section]

\newtheorem{cor}[thm] {Corollary}
\newtheorem{lem} [thm]{Lemma}

\newtheorem{defn}[thm]{Definition}
\numberwithin{equation}{section}

\newcommand\diag{\operatorname{diag}}
\def\ni{\noindent}
\begin{document}
\label{'ubf'}
\setcounter{page}{1} 

\markboth {\hspace*{-9mm} \centerline{\footnotesize \sc
    On Two Laplacian Matrices for Skew Gain Graphs}
                 }
                { \centerline {\footnotesize \sc
                    Roshni T Roy Shahul Hameed K Germina K A  
              } \hspace*{-9mm}
               }
\begin{center}
{
       {\huge \textbf{On Two Laplacian Matrices \\ for Skew Gain Graphs
                               }
       }
\\

\medskip
Roshni T Roy \footnote{\small Department of Mathematics, Central University of Kerala, Kasaragod - 671316,\ Kerala,\ India.\ Email:roshnitroy@gmail.com}
Shahul Hameed K \footnote{\small  Department of Mathematics, K M M Government\ Women's\ College, Kannur - 670004,\ Kerala,  \ India.  E-mail: shabrennen@gmail.com} 
Germina K A \footnote{\small  Department of Mathematics, Central University of Kerala, Kasaragod - 671316,\ Kerala,\ India.\ Email: srgerminaka@gmail.com}

}
\end{center}
\newcommand\NEPS{\operatorname{NEPS}}
\thispagestyle{empty}
\begin{abstract}
\ni Let $G=(V,\overrightarrow{E})$ be a graph with some prescribed orientation for the edges and $\Gamma$ be an arbitrary group. If $f\in \mathrm{Inv}(\Gamma)$ be an anti-involution then the \emph{skew gain graph} $\Phi_f=(G,\Gamma,\varphi,f)$ is such that the \emph{skew gain function} $\varphi:\overrightarrow{E}\rightarrow \Gamma$ satisfies $\varphi(\overrightarrow{vu})=f(\varphi(\overrightarrow{uv}))$. In this paper, we study two different types, Laplacian and $g$-Laplacian matrices for a skew gain graph where the skew gains are taken from the  multiplicative group $F^\times$ of a field $F$ of characteristic zero. Defining incidence matrix, we also prove the matrix tree theorem for skew gain graphs in the case of the $g$-Laplacian matrix.

---------------------------------------------------------------------------------------------\\
\end{abstract}
\section{Introduction}

Throughout this article, unless otherwise mentioned, by a graph we mean a finite, connected, simple graph and any terms which are not mentioned here, the reader may refer to \cite{fh}.\\
A gain graph is a graph with some orientation for the edges such that each edge has a gain, that is a label from a group so that reversing the direction of edge inverts the gain \cite{tz}. Generalizing the notion of gain graphs, skew gain graphs are defined such that gain of an edge (we call it as skew gain) is related to the skew gain of the reverse edge by an anti-involution \cite{jh}. The general expression for computing the coefficients of the characteristic polynomial of the adjacency matrix of skew gain graphs are studied in \cite{sh}. Laplacian matrix of a graph and matrix tree theorem are well studied by many which can be referred to for instance from \cite{rm}.  The matrix tree theorem for signed graph can be seen in Zaslavsky \cite{tz2} and on a more general setting in Chaiken \cite{sc}. In this paper, we define two different Laplacian matrices for skew gain graphs and prove matrix tree theorem for skew gain graphs.\\
Let $\Gamma$ be an arbitrary group. A function $f:\Gamma\rightarrow \Gamma$ is an \emph{involution} if $f(f(x))=x$ for all $x\in \Gamma$.  A function $f:\Gamma\rightarrow \Gamma$ is called an \emph{anti-homomorphism} if $f(xy)=f(y)f(x)$ for all $x,y\in\Gamma$. For an abelian group an anti-homomorphism is always a homomorphism. An involution $f:\Gamma\rightarrow \Gamma$ which is an anti-homomorphism is called an \emph{anti-involution}. We use $\mathrm{Inv}(\Gamma)$ to denote the set of all anti-involutions on $\Gamma.$ We define  $g:\Gamma\rightarrow \Gamma$ such that $g(x)=xf(x)$ for all $x\in \Gamma.$ 
\begin{defn}[\cite{jh}]
	\rm{Let $G=(V,\overrightarrow{E})$ be a graph with some prescribed orientation for the edges and $\Gamma$ be an arbitrary group. If $f\in \mathrm{Inv}(\Gamma)$ be an anti-involution then the \emph{skew-gain graph} $\Phi_f=(G,\Gamma,\varphi,f)$ is such that the \emph{skew gain function} $\varphi:\overrightarrow{E}\rightarrow \Gamma$ satisfies $\varphi(\overrightarrow{vu})=f(\varphi(\overrightarrow{uv}))$.}
\end{defn}
The adjacency matrix of a skew gain graph is defined when the skew gains are taken from the  multiplicative group $F^\times$ where $F$ is a field of characteristic zero. Here $f \in \mathrm{Inv}(\Gamma)$ is an involutive automorphism. We use the notation $u\sim v$ when the vertices $u$ and $v$ are adjacent and similar notation for the incidence of an edge on a vertex.
\begin{defn}[\cite{sh}]
	\rm{Given a skew gain graph $\Phi_f=(G,F^\times,\varphi,f)$ its adjacency matrix $A(\Phi_f)=(a_{ij})_n$ is defined as the square matrix of order $n=|V(G)|$ where \\
	$a_{ij} =
	\left\{
	\begin{array}{ll}
	\varphi(v_iv_j)  & \mbox{if } v_i\sim v_j \\
	0 & \mbox{otherwise }
	\end{array}
	\right.$ \\ such that whenever $a_{ij}\neq 0$,  
	$a_{ji}=f(a_{ij})$.} 
\end{defn}
In the following sections, we define Laplacian matrix and $g$-Laplacian matrix of skew gain graphs by defining the corresponding degree and $g$-degree matrices. We also define the incidence matrix of a skew gain graph.
 
\section{Laplacian matrix for skew gain graphs}

\begin{defn}
\rm{	Given a skew gain graph $\Phi_f=(G,F^\times,\varphi,f)$, the degree of the vertex  $v_i$, $d(v_i)$, is obtained by adding the multiplicative identity of the field $F^\times, d_i$ times, where $d_i$ is the degree of the vertex $v_i$ in the underlying graph $G.$\\
	Degree matrix $D(\Phi_f)$ can be defined as the diagonal matrix $(d_{ij})_n$ where $d_{ii} = d(v_i).$}	
\end{defn}

\begin{defn}
\rm{	Given a skew gain graph $\Phi_f=(G,F^\times,\varphi,f)$ its Laplacian matrix is defined as $L(\Phi_f)=D(\Phi_f) - A(\Phi_f).$ We define the Laplacian charactersitic polynomial of the skew gain graph $\Phi_f$ as $det(xI-L(\Phi_f)).$ The eigenvalues of the Laplacian matrix, counting the multiplicities, of a skew gain graph are called the Laplacian eigenvalues or Laplacian spectra of that skew gain graph.}	
\end{defn}

\begin{lem}\label{l1}\cite{nb}
   Let $A=(a_{i,j})$ be an $n \times n$ matrix. Then determinant of $A$ has the expansion 
   \begin{equation*}
     \det(A) = \sum sgn(\pi) a_{1,\pi (1)}a_{2,\pi (2)} \dots a_{n,\pi (n)}
   \end{equation*}
    where the summation is over all permutations $\pi$ on the set \{1,2,3, \dots, n\} and $sgn(\pi)$ is the sign of the permutation $\pi.$ If $\pi$ is an even cycle, then $sgn(\pi)= -1$ and if $\pi$ is an odd cycle, then $sgn(\pi)= +1.$ Thus the sign of an arbitrary permutation $\pi$ is $(-1)^{N_e}$, where $N_e$ is the number of even cycles in cyclic representation of $\pi.$ 
\end{lem}

Let $\mathfrak{L}(G)$ denotes set of all elementary subgraphs $L$ of $G$ (of all orders) and $K_e(L)$ denotes the number of components in $L$ having even order. Also let $\mathcal{M}(G)$ denotes the set of all matchings $M$ in the graph $G$ and $K(M)$ denotes the number of edges in $M.$ We denote the Laplacian characteristic polynomial of skew gain graph by $\chi(\Phi_f,x)= \det(xI-L(\Phi_f)).$

\begin{thm}\label{thmch}
	If $\Phi_f =(G,F^\times,\varphi,f)$ is a skew gain graph where $G=(V,\overrightarrow{E})$ is a graph of order $n,$	then 
	\begin{align*}
	\chi(\Phi_f,x)&= \displaystyle \prod_{v \in V(G)}(x-d(v)) + \\
      &\qquad	\displaystyle \sum_{L\in\mathfrak{L}(G)}(-1)^{K_e(L)} \displaystyle\prod_{K_2 \in L}g(\varphi(\overrightarrow{e}))\displaystyle \prod_{C \in L}(\varphi(C)+f(\varphi(C)))\displaystyle\prod_{v \notin V(L)}(x-d(v)).
	\end{align*}	
\end{thm}
\begin{proof}
	Let $d_i$ denotes the degree of the vertex $v_i$ and let the adjacency matrix of skew gain graph $\Phi_f$ be
	 $A(\Phi_f)= \begin{pmatrix} 
		  0 & a_{1,2}  & a_{1,3} & \dots & a_{1,n}\\
		  a_{2,1} & 0 & a_{2,3} & \dots & a_{2,n}\\
		  \dots \\
		  a_{n,1} & a_{n,2} & a_{n,3} & \dots & 0
	\end{pmatrix}$.\\
    Then Laplacian characteristic polynomial of skew gain graph $\Phi_f$ is \begin{equation*}
        \chi(\Phi_f,x)= \det(xI-L(\Phi_f)) =
         \det \begin{pmatrix} 
    	    x-d_1 & a_{1,2} & a_{1,3} & \dots & a_{1,n}\\
    	    a_{2,1} & x-d_2 & a_{2,3} & \dots & a_{2,n}\\
    	    \dots \\
    	    a_{n,1} & a_{n,2} & a_{n,3} & \dots & x-d_n
          \end{pmatrix}.
    \end{equation*}
     Using Lemma \ref{l1}, corresponding to the identity permutation, we get the term $\displaystyle \prod_{v \in V(G)}(x-d(v)).$ Now, for any non-identity permutation $\pi,$ consider the term $sgn(\pi)a_{1,\pi (1)}a_{2,\pi (2)} \dots a_{n,\pi (n)}.$ Any permutation $\pi$ can be expressed as a product of disjoint cycles. Thus if $\pi$ fixes the $i^{th}$ element, $a_{i,i} = x-d(v_i).$ Now a cycle $(ij)$ of length two in $\pi$ corresponds to $a_{i,j}.a_{j,i}$ which corresponds to the edges $\overrightarrow{v_iv_j}$ and $\overrightarrow{v_jv_i}$ in $G.$ Any cycle $(pqr \dots t)$ of length greater than $2$ corresponds to $a_{p,q}a_{q,r} \dots a_{t,p}$ which gives a cycle $v_pv_qv_r \dots v_tv_p$ in $G.$ Thus, corresponding to the non-identity permutation $\pi$ we get an elementay subgraph $L$ of $G$ and $a_{1,\pi (1)}a_{2,\pi (2)} \dots a_{n,\pi (n)}$ becomes $\displaystyle\prod_{K_2 \in L}g(\varphi(\overrightarrow{e}))\displaystyle \prod_{C \in L}(\varphi(C)+f(\varphi(C))\displaystyle\prod_{v \notin V(L)}(x-d(v)).$\\
     Now, $sgn(\pi) = (-1)^{N_e},$ where $N_e$ is the number of even cycles in $\pi,$ which is same as the number of components in $L$ having even order.   
\end{proof}
When the underlying graph of $\Phi_f =(G,F^\times,\varphi,f)$ is a cycle or path, we call it as a skew gain cycle or skew gain path respectively.
\begin{cor}
	If $\Phi_f=(P_n,F^\times,\varphi,f)$ is a  skew gain path, then its Laplacian characteristic polynomial is
	\begin{equation*}
	\chi(\Phi_f, x)= (x-2)^{n-2}(x-1)^2+  \displaystyle\sum_{M\in \mathcal{M}(P_n)}(-1)^{K(M)}\prod_{\overrightarrow{e}\in M}g(\varphi(\overrightarrow{e}))\displaystyle \prod_{v \notin V(M)}(x-d(v)).
	\end{equation*}	
\end{cor}

\begin{cor}
	If $\Phi_f=(C_n,F^\times,\varphi,f)$ is a skew gain cycle, then its Laplacian characteristic polynomial is
	\begin{align*}
	\chi(\Phi_f, x)&= (x-2)^n +(-1)^{n-1} (\varphi(C_n) + f(\varphi(C_n))) + \\
	&\qquad \displaystyle\sum_{M\in \mathcal{M}(C_n)}(-1)^{K(M)}\prod_{\overrightarrow{e}\in M}g(\varphi(\overrightarrow{e}))\displaystyle \prod_{v \notin V(M)}(x-d(v)). 
	\end{align*}	
\end{cor}
\begin{proof}
     The only elementary subgraph $L\in\mathfrak{L}(C_n)$ containing cycle as a component is $C_n$ itself. If $n$ is even then $(-1)^{K_e(L)} = -1 = (-1)^{n-1}$ and if $n$ is odd $(-1)^{K_e(L)} = (-1)^0 = 1 = (-1)^{n-1}.$ All other elementary subgraphs contains $K_2$ as components which can be considered as matchings in $C_n$ and hence using theorem \ref{thmch} we get
     \begin{align*}
     \chi(\Phi_f, x)&= (x-2)^n +(-1)^{n-1} (\varphi(C_n) + f(\varphi(C_n))) + \\
     &\qquad \displaystyle\sum_{M\in \mathcal{M}(C_n)}(-1)^{K(M)}\prod_{\overrightarrow{e}\in M}g(\varphi(\overrightarrow{e}))\displaystyle \prod_{v \notin V(M)}(x-d(v)).
     \end{align*} 
\end{proof}
\begin{cor}\label{str}
	If $\Phi_f=(K_{1,n},F^\times,\varphi,f)$ is a  skew gain graph with underlying graph as the star $K_{1,n}$, then its Laplacian characteristic polynomial is
	\begin{equation*}
	\chi(\Phi_f, x)= (x-1)^n(x-n) - (x-1)^{(n-1)} \displaystyle\sum_{\overrightarrow{e} \in E(K_{1,n})} g(\varphi(\overrightarrow{e})).
	\end{equation*}	
\end{cor}

Now we move to the Laplacian spectra of some particular classes of skew gain graphs. 
\begin{thm}\label{reg}
	If $\Phi_f=(G,F^\times,\varphi,f)$ is a skew gain graph where $G$ is $d$-regular, then the Laplacian eigenvalues of $L(\Phi_f)$ are $d- \lambda$ where $\lambda$ is an eigenvalue of its adjacency matrix $A(\Phi_f).$
\end{thm}
\begin{proof}
	If $d(v_i) = d,$ for all vertices $v_i$ in $\Phi_f=(G,F^\times,\varphi,f)$, then its Laplacian matrix is $L(\Phi_f)= dI - A(\Phi_f),$ which implies eigenvalues are  $d- \lambda$ where $\lambda$ is an eigenvalue of $A(\Phi_f).$
\end{proof}

We define for a matrix $B=(a_{ij})\in M_{m\times n}(F), B^{f}=(b_{ij})\in M_{m\times n}(F)$ where $f\in \mathrm{Inv}(F^{\times})$ as 
$b_{ij} =
\left\{
\begin{array}{ll}
f(a_{ij})  & \mbox{if } a_{ij}\neq 0 \\
0 & \mbox{otherwise. }
\end{array}
\right.$\\

\begin{thm}
    Let $\Phi_f =(G,F^\times,\varphi,f)$ be a skew gain graph where $G= K_{m,m}$ is a complete bipartite graph. Then the eigenvalues of $L(\Phi_f)$ are $m-\lambda$ such that $\lambda^2$ is an eigenvalue of $B(B^f)^T.$
\end{thm}
\begin{proof}
	The adjacency eigenvalues of $\Phi_f =(G,F^\times,\varphi,f),$ where $G= K_{m,m}$ is a complete bipartite graph, are $\lambda$ such that $\lambda^2$ is an eigenvalue of $B(B^f)^T$ \cite{sh}. Hence by Theorem \ref{reg}, since $\Phi_f$ is regular with degree $m,$ we get the eigenvalues of $L(\Phi_f)$ are $m-\lambda$ such that $\lambda^2$ is an eigenvalue of $B(B^f)^T.$
\end{proof}

\begin{thm}
	Let $\Phi_f =(G,F^\times,\varphi,f)$ be a skew gain graph where $G= K_{1,n}$ is a star of order $n+1.$ Then $\det(L(\Phi_f)) = n-\displaystyle\sum_{\overrightarrow{e} \in E(G)} g(\varphi(\overrightarrow{e})).$
\end{thm}
\begin{proof}
When we put $x=0$ in the characteristic polynomial of $L(\Phi_f)$, in Corollary \ref{str}, we get the constant term in the polynomial as $(-1)^{n-1} \big( n-\displaystyle\sum_{\overrightarrow{e} \in E(G)} g(\varphi(\overrightarrow{e})) \big),$ which is equal to  $(-1)^{n+1}\det(L(\Phi_f)).$ Hence $\det(L(\Phi_f)) = n-\displaystyle\sum_{\overrightarrow{e} \in E(G)} g(\varphi(\overrightarrow{e})).$ 
\end{proof}

\begin{thm}
	If $\Phi_f=(G,F^\times,\varphi,f)$ is a skew gain graph where $G= K_{1,n}$ is a star of order $n+1$, then the Laplacian spectrum of $\Phi_f$ is
	 $\begin{pmatrix}  \frac{n+1+\sqrt{(n+1)^2-4 \det(L(\Phi_f))}}{2} & \frac{n+1-\sqrt{(n+1)^2-4\det(L(\Phi_f))}}{2}  & 1\\
	1 & 1 & n-1
	\end{pmatrix}.$
\end{thm}
\begin{proof}
	The Laplacian matrix of skew gain graph $\Phi_f =(G,F^\times,\varphi,f),$ where $G= K_{1,n},$ is
	 $L(\Phi_f)= \begin{pmatrix} 
	  n & l_{1,2}  & l_{1,3} & \dots & l_{1,n+1}\\
	  l_{2,1} & 1 & 0 & \dots & 0\\
	  \dots \\
	  l_{n+1,1} & 0 & 0 & \dots & 1
	\end{pmatrix}$. 
	By Corollary \ref{str}, Laplacian characteristic polynomial of $\Phi_f$ is 
	\begin{eqnarray*}
	\chi(\Phi_f, x)= (x-1)^{n-1}\big((x-n)(x-1) -  \displaystyle\sum_{\overrightarrow{e} \in E(G)} g(\varphi(\overrightarrow{e}))\big)\\
	= (x-1)^{n-1}\big((x^2 -(n+1)x + n -  \displaystyle\sum_{\overrightarrow{e} \in E(G)} g(\varphi(\overrightarrow{e}))\big) 
	\end{eqnarray*}
	From this the Laplacian spectrum of $\Phi_f$ becomes \\
	  $\begin{pmatrix}  \frac{n+1+\sqrt{(n+1)^2-4 \det(L(\Phi_f))}}{2} & \frac{n+1-\sqrt{(n+1)^2-4\det(L(\Phi_f))}}{2}  & 1\\
	 1 & 1 & n-1
	 \end{pmatrix}$.	
\end{proof}

\section{$g$-Laplacian matrix for skew gain graphs}

Now we define the $g$-Laplacian matrix of a skew gain graph as follows:\\
For an oriented edge $\overrightarrow{e_j}=\overrightarrow{v_iv_k}$ we take $v_i$ as the tail of that edge and $v_k$ as its head and we write $t(\overrightarrow{e_j})=v_i$ and $h(\overrightarrow{e_j})=v_k$.
\begin{defn}
	\rm{Given a skew gain graph $\Phi_f=(G,F^\times,\varphi,f)$ its $g$-Laplacian matrix is defined as $L_g(\Phi_f)=D_g(\Phi_f) - A(\Phi_f)$ where the diagonal matrix $D_g(\Phi_f)$ is $\diag\Big(\displaystyle\sum_{\overrightarrow{e}:v_i\sim \overrightarrow{e}}\sqrt{g(\varphi(\overrightarrow{e})}\Big)$ where $\sqrt{a}$ for $a\in F$ belongs to the algebraic closure of the field $F$.}
	The matrix $D_g(\Phi_f)$ is the $g$-degree matrix of $\Phi_f.$
\end{defn}
The incidence matrix for a skew gain graph $\Phi_f$ can be defined as follows 
\begin{defn}\label{incid}
	\rm{Given a skew gain graph $\Phi_f=(G,F^\times,\varphi,f)$ its (oriented) incidence matrix is defined as $\mathrm{H}(\Phi_f)=(b_{ij})$ where 
	$$b_{ij} =
	\begin{cases}
	g(\varphi(\overrightarrow{e_j}))  & \mbox{if } t(\overrightarrow{e_j})=v_i ,\\
	-f(\varphi(\overrightarrow{e_j}))\sqrt{g(\varphi(\overrightarrow{e_j}))}  & \mbox{if } h(\overrightarrow{e_j})=v_i ,\\
	0 & \mbox{otherwise. }
	\end{cases}
	$$ } 
\end{defn}

Clearly the definitions of Laplacian, $g$-Laplacian and incidence matrix of a skew gain graph coincide with the corresponding definitions for ordinary graphs, signed graphs and gain graphs which are extensively studied in \cite{nb,rm,shkg,tz2}. 
Now we define a matrix operation for the incidence matrix $\mathrm{H}(\Phi_f)$ as follows:\\
$\mathrm{H}^{\#}$ is the transpose of the matrix obtained by replacing each column element as under:\\
(i) $g(\varphi(\overrightarrow{e_j}))$ replaced by $ (\sqrt{g(\varphi(\overrightarrow{e_j}))})^{-1}$ and\\
(ii) $-f(\varphi(\overrightarrow{e_j}))\sqrt{g(\varphi(\overrightarrow{e_j}))}$ replaced by $-(f(\varphi(\overrightarrow{e_j})))^{-1}$

\begin{thm}\label{thm1}
   For a skew gain graph $\Phi_f=(G,F^\times,\varphi,f),$ $L_g(\Phi_f)=\mathrm{H}(\Phi_f)\mathrm{H}^{\#}(\Phi_f).$
\end{thm}
\begin{proof} 
	Let $v_1,v_2, \dots , v_n$ and $\overrightarrow{e_1},\overrightarrow{e_2}, \dots ,\overrightarrow{e_m}$ be the vertices and edges in $G$, respectively. Denoting $\mathrm{H}(\Phi_f)$ by $\Big(\eta_{v_i\overrightarrow{e_j}}\Big)$ and $\mathrm{H}^{\#}(\Phi_f)$ by $\Big(\eta'_{\overrightarrow{e_i}v_j}\Big)$, let the $i^{th}$ row vector of $\mathrm{H}(\Phi_f)$ be $[\eta_{v_i\overrightarrow{e_1}}, \eta_{v_i\overrightarrow{e_2}},\dots ,\eta_{v_i\overrightarrow{e_m}}]$ and $j^{th}$ column of $\mathrm{H}^{\#}(\Phi_f)$ be $[\eta'_{\overrightarrow{e_1}v_J}, \eta'_{\overrightarrow{e_2}v_j},\dots ,\eta'_{\overrightarrow{e_m}v_j}]$. Now the $(i,j)^{th}$ entry of $\mathrm{H}\mathrm{H}^{\#}$ is $\displaystyle \sum_{k=1}^{m}\eta_{v_i\overrightarrow{e_k}}\eta'_{\overrightarrow{e_k}v_j}$.
	
	For $i=j$, $\eta_{v_i\overrightarrow{e_k}}\eta'_{\overrightarrow{e_k}v_j} \neq 0$ if and only if $\overrightarrow{e_k}$ is incident to $v_i$. If $t(\overrightarrow{e_k})=v_i$ then $\eta_{v_i\overrightarrow{e_k}}=g(\varphi(\overrightarrow{e_k}))$ in $\mathrm{H}(\Phi_f)$ and hence $\eta'_{\overrightarrow{e_k}v_j}=\sqrt{g(\varphi(\overrightarrow{e_k}))^{-1}}$ in $\mathrm{H}^{\#}(\Phi_f)$ so that $\eta_{v_i\overrightarrow{e_k}}\eta'_{\overrightarrow{e_k}v_j}= \sqrt{g(\varphi(\overrightarrow{e_k})}$ in $\mathrm{H}(\Phi_f)\mathrm{H}^{\#}(\Phi_f)$. If $h(\overrightarrow{e_k})=v_i$ then  $\eta_{v_i\overrightarrow{e_k}}=-f(\varphi(\overrightarrow{e_k}))\sqrt{g(\varphi(\overrightarrow{e_k}))}$ and hence  $\eta'_{\overrightarrow{e_k}v_j}=-f(\varphi(\overrightarrow{e_k}))^{-1}$ so that $\eta_{v_i\overrightarrow{e_k}}\eta'_{\overrightarrow{e_k}v_j}= \sqrt{g(\varphi(\overrightarrow{e_k})}$. Thus the diagonal entries in $\mathrm{H}(\Phi_f)\mathrm{H}^{\#}(\Phi_f)$ is $\displaystyle\sum_{\overrightarrow{e}:v_i\sim \overrightarrow{e}}\sqrt{(g(\varphi(\overrightarrow{e}))}.$
	
	For $i \neq j$, $\eta_{v_i\overrightarrow{e_k}}\eta'_{\overrightarrow{e_k}v_j}\neq 0$ if and only if $\overrightarrow{e_k}$ is an edge joining $v_i$ and $v_j$.
	If $\overrightarrow{e_k} = \overrightarrow{v_iv_j}$ then $\eta_{v_i\overrightarrow{e_k}}\eta'_{\overrightarrow{e_k}v_j} = g(\varphi(\overrightarrow{e_k})).(-f(\varphi(\overrightarrow{e_k}))^{-1})= - \varphi(\overrightarrow{e_k})$ and if $\overrightarrow{e_k} = \overrightarrow{v_jv_i}$ then $\eta_{v_i\overrightarrow{e_k}}\eta'_{\overrightarrow{e_k}v_j} = -f(\varphi(\overrightarrow{e_k}))\sqrt{g(\varphi(\overrightarrow{e_k}))}.\sqrt{g(\varphi(\overrightarrow{e_k}))^{-1}} = - f(\varphi(\overrightarrow{e_k}))$.\\
	In both cases, the $(i,j)^{th}$ entry of $\mathrm{H}(\Phi_f)\mathrm{H}^{\#}(\Phi_f)$ coincides with the $(i,j)^{th}$ entry of $L_g(\Phi_f)$ and hence the proof.
\end{proof}

From the definition~\ref{incid}, we will have the following deductions:\\
(i) In the case of real weighted graphs where $f(x)=x$ so that $g(x)=x^2$ ( which is a particular skew gain graphs which we can be used to deal with weighted signed graphs also), the incidence matrix $\mathrm{H}=(b_{ij})$ has
$$b_{ij} =\begin{cases}
w(\overrightarrow{e_j})^2  & \mbox{if } t(\overrightarrow{e_j})=v_i \\
-w(\overrightarrow{e_j})^2  & \mbox{if } h(\overrightarrow{e_j})=v_i ,\\
0 & \mbox{otherwise. }
\end{cases}
$$
(ii) In the case of complex skew gain graph with $f(z)=\overline{z}$, so that $g(z)=|z|^2$, the incidence matrix $\mathrm{H}=(b_{ij})$ has
$$b_{ij} =\begin{cases}
|w(\overrightarrow{e_j})|^2  & \mbox{if } t(\overrightarrow{e_j})=v_i ,\\
-\overline{w(\overrightarrow{e_j})}|w(\overrightarrow{e_j})|  & \mbox{if } h(\overrightarrow{e_j})=v_i ,\\
0 & \mbox{otherwise. }
\end{cases}
$$

\begin{lem}\label{lem1}\cite{jgb}
	Let $A$ be an $m \times n$ matrix and  $B$ be an $n \times k$ matrix. Then $rank(AB)  \leq min\{ rank(A), rank(B)\}.$ Also $rank(A) \leq min\{m,n\}.$
\end{lem}

\begin{lem}\label{lem2}\cite{jgb}
	If $A \in M_n(F)$ is a block triangular matrix of the form\\
	$A= \begin{pmatrix} 
	A_{11} & A_{12}  & \dots & A_{1k}\\
	0 & A_{22} &  \dots & A_{2k}\\
	\dots &\dots & \dots\\
	0 & 0 & \dots & A_{kk}
	\end{pmatrix}$ where each $A_{ii}$ is a square matrix and the $0’s$ are zero matrices of appropriate
	size, then $\det(A) = \displaystyle \prod_{i=1}^k \det(A_{ii})$.
\end{lem}

\begin{thm}\label{tree}
	If $\Phi_f =(G,F^\times,\varphi,f)$ is a skew gain graph, where $G$ is a tree of order $n,$ then $\det L_g(\Phi_f) =0.$
\end{thm}
\begin{proof}
	A tree on $n$ vertices have $n-1$ edges. Thus the incidence matrix $H(\Phi_f)$ has order $n \times n-1$. Now, by Lemma \ref{lem1}, rank($\mathrm{H}(\Phi_f)\mathrm{H}^{\#}(\Phi_f)$) is less than or equal to $n-1$ which implies $\det (\mathrm{H}(\Phi_f)\mathrm{H}^{\#}(\Phi_f)) =0.$ Thus by Theorem \ref{thm1},
	\begin{equation*}
	\det(L_g(\Phi_f)) = \det(\mathrm{H}(\Phi_f)\mathrm{H}^{\#}(\Phi_f)) = 0.
	\end{equation*}
\end{proof}

\begin{thm}\label{cycle}
	If $\Phi_f =(C_n,F^\times,\varphi,f)$ is a skew gain cycle then $\det L_g(\Phi_f) = 2\sqrt{\displaystyle \prod_{\overrightarrow{e} \in E(C_n)} g(\varphi(\overrightarrow{e}))} -[\varphi(C_n) + f(\varphi(C_n))].$
\end{thm}
\begin{proof}
	Let the skew gain cycle be $C_n = v_1\overrightarrow{e_1}v_2\overrightarrow{e_2}v_3\overrightarrow{e_3} \dots v_{n-1}\overrightarrow{e_{n-1}}v_n\overrightarrow{e_n}v_1$. Its incidence matrix $\mathrm{H}$ is \\
	
	$ \begin{pmatrix} 
	g(\varphi(\overrightarrow{e_1}))  & 0  &  \dots & 0 & -f(\varphi(\overrightarrow{e_n}))\sqrt{g(\varphi(\overrightarrow{e_n}))}\\
	-f(\varphi(\overrightarrow{e_1}))\sqrt{g(\varphi(\overrightarrow{e_1}))} & g(\varphi(\overrightarrow{e_2})) & \dots & 0 & 0\\
	\dots &\dots & \dots &\dots &\dots\\
	0 & 0 & \dots & g(\varphi(\overrightarrow{e_{n-1}})) & 0 \\
	0 & 0 & \dots & -f(\varphi(\overrightarrow{e_{n-1}}))\sqrt{g(\varphi(\overrightarrow{e_{n-1}}))} & g(\varphi(\overrightarrow{e_n}))
	\end{pmatrix}$. \\\\
	Expanding along the first row to find the determinant of $\mathrm{H}$, we get \\
	$ \det (\mathrm{H}) = g(\varphi(\overrightarrow{e_1})) M_{1,1}  + (-1)^n f(\varphi(\overrightarrow{e_n}))\sqrt{g(\varphi(\overrightarrow{e_n}))} M_{1,n}$, where \\
	
	$ M_{1,1} = \det \begin{pmatrix} 
	g(\varphi(\overrightarrow{e_2})) & \dots & 0 & 0\\
	\dots \\
	0 & \dots & g(\varphi(\overrightarrow{e_{n-1}})) & 0 \\
	0 & \dots & -f(\varphi(\overrightarrow{e_{n-1}}))\sqrt{g(\varphi(\overrightarrow{e_{n-1}}))} & g(\varphi(\overrightarrow{e_n}))
	\end{pmatrix}$\\
	
	$M_{1,n} = \det \begin{pmatrix} 
	-f(\varphi(\overrightarrow{e_1}))\sqrt{g(\varphi(\overrightarrow{e_1}))} & g(\varphi(\overrightarrow{e_2})) & \dots & 0 \\
	\dots \\
	0 & 0 & \dots & g(\varphi(\overrightarrow{e_{n-1}}))  \\
	0 & 0 & \dots & -f(\varphi(\overrightarrow{e_{n-1}}))\sqrt{g(\varphi(\overrightarrow{e_{n-1}}))} 
	\end{pmatrix}$
	
	Clearly $M_{1,1}$ and $M_{1,n}$ are determinant of triangular matrices and hence it is the product of the diagonal entries. Thus \\
	$M_{1,1}= g(\varphi(\overrightarrow{e_2}))g(\varphi(\overrightarrow{e_3})) \dots g(\varphi(\overrightarrow{e_n}))$ and\\ 
	$M_{1,n} = (-1)^{n-1} f(\varphi(\overrightarrow{e_1}))\sqrt{g(\varphi(\overrightarrow{e_1}))}f(\varphi(\overrightarrow{e_2}))\sqrt{g(\varphi(\overrightarrow{e_2}))} \dots f(\varphi(\overrightarrow{e_{n-1}}))\sqrt{g(\varphi(\overrightarrow{e_{n-1}}))}$.\\
	Hence $\det (\mathrm{H}) = \displaystyle \prod_{\overrightarrow{e} \in E(C_n)} g(\varphi(\overrightarrow{e})) - f(\varphi(C_n))\sqrt{\displaystyle \prod_{\overrightarrow{e} \in E(C_n)} g(\varphi(\overrightarrow{e}))}.$ \\
	
	Now considering the matrix $\mathrm{H}^{\#},$\\
	 $\mathrm{H}^{\#}=\begin{pmatrix} 
	\sqrt{g(\varphi(\overrightarrow{e_1}))}^{-1}  & 0  &  \dots & 0 & -f(\varphi(\overrightarrow{e_n}))^{-1}\\
	-f(\varphi(\overrightarrow{e_1}))^{-1} & \sqrt{g(\varphi(\overrightarrow{e_2}))}^{-1} & \dots & 0 & 0\\
	\dots \\
	0 & 0 & \dots & \sqrt{g(\varphi(\overrightarrow{e_{n-1}}))}^{-1} & 0 \\
	0 & 0 & \dots & -f(\varphi(\overrightarrow{e_{n-1}}))^{-1} & \sqrt{g(\varphi(\overrightarrow{e_n}))}^{-1} 
	\end{pmatrix}.$ \\
	
	Finding its determinant in a similiar way we get 
	\begin{eqnarray*}
		\det (\mathrm{H}^{\#}) = \displaystyle \prod_{\overrightarrow{e} \in E(C_n)} \sqrt{g(\varphi(\overrightarrow{e}))}^{-1} - f(\varphi(C_n))^{-1} \\
		= \sqrt{ \displaystyle \prod_{\overrightarrow{e} \in E(C_n)} g(\varphi(\overrightarrow{e}))}^{-1} - f(\varphi(C_n))^{-1}.
	\end{eqnarray*}
	
	Now from Theorem~\ref{thm1} $L_g=\mathrm{H}\mathrm{H}^{\#}$ which gives $\det(L_g) = \det (\mathrm{H}) \det \mathrm{H}^{\#}.$
	Thus $\det L_g(\Phi_f) = 2\sqrt{\displaystyle \prod_{\overrightarrow{e} \in E(C_n)} g(\varphi(\overrightarrow{e}))} -[\varphi(C_n) + f(\varphi(C_n))].$
\end{proof}

\begin{thm}\label{unicyclic}
		If $\Phi_f =(G,F^\times,\varphi,f)$ is a skew gain graph of order $n$ where $G$ is a unicyclic graph with unique cycle $C$ then\\ $\det L_g(\Phi_f) =\sqrt{ \displaystyle \prod_{\overrightarrow{e} \notin E(C)} g(\varphi(\overrightarrow{e}))}\Big(2\sqrt{\displaystyle \prod_{\overrightarrow{e} \in E(C)} g(\varphi(\overrightarrow{e}))} -[\varphi(C) + f(\varphi(C))]\Big).$
\end{thm}
\begin{proof}
	Let $C = v_1\overrightarrow{e_1}v_2\overrightarrow{e_2} \dots v_p\overrightarrow{e_p}v_1$ be the unique cycle and define the orientation of edges as for $i<j$ the edge $\overrightarrow{e_{i,j}}$ has tail $t(\overrightarrow{e_{i,j}})=v_i$ and head $h(\overrightarrow{e_{i,j}})=v_j$. 
	We get the incidence matrix $\mathrm{H}(\Phi_f)$ as an upper triangular block martix with diagonal blocks $A_1, A_2, \dots , A_k, k=n-p+1$ where $A_1$ corresponds to the vertices and edges in the cycle $C$ and $A_i,i = 2,3,\dots n-p+1$ are one element matrices $[-f(\varphi(\overrightarrow{e}))\sqrt{g(\varphi{(\overrightarrow{e})})}]$ corresponding to the edges $\overrightarrow{e}$ not in $C$. Then, by Lemma \ref{lem2}, $\det(H(\Phi_f)) = \prod \det( A_i). $ Now using Theorem \ref{cycle},  $\det(\mathrm{H}) =\displaystyle \prod_{\overrightarrow{e} \notin E(C)}(-f(\varphi(\overrightarrow{e}))(\sqrt{g(\varphi{(\overrightarrow{e})})})\big(\displaystyle \prod_{\overrightarrow{e} \in E(C)} g(\varphi(\overrightarrow{e})) - f(\varphi(C))\sqrt{\displaystyle \prod_{\overrightarrow{e} \in E(C)} g(\varphi(\overrightarrow{e}))}\big). $\\
	
	Similiarly we get \\ $\det(\mathrm{H^{\#}}) =\displaystyle \prod_{\overrightarrow{e} \notin E(C)}(-f(\varphi(\overrightarrow{e}))^{-1})\big( \sqrt{ \displaystyle \prod_{\overrightarrow{e} \in E(C)} g(\varphi(\overrightarrow{e}))}^{-1} - f(\varphi(C))^{-1} \big). $\\
	
	Since $\det L_g(\Phi_f) =\det {\mathrm{H}\mathrm{H}^{\#}}$ we get, \\ $\det L_g(\Phi_f) =\sqrt{ \displaystyle \prod_{\overrightarrow{e} \notin E(C)} g(\varphi(\overrightarrow{e}))}\Big(2\sqrt{\displaystyle \prod_{\overrightarrow{e} \in E(C)} g(\varphi(\overrightarrow{e}))} -[\varphi(C) + f(\varphi(C))]\Big).$
\end{proof}

A $1$-tree is a connected unicyclic graph and a $1$-forest is a disjoint union of $1$-trees. A spanning subgraph of $G$ which is a $1$-forest is called as an essential spanning subgraph of $G$. We denote the collection of all essential spanning subgraphs of $G$ by $\mathfrak{E}(G)$

\begin{thm}\label{forest}
	If $\Phi_f =(G,F^\times,\varphi,f)$ is a skew gain graph where $G$ is a $1$-forest, then $\det L_g(\Phi_f) =\displaystyle \prod_{\Psi \in G} \sqrt{ \displaystyle \prod_{\overrightarrow{e} \notin C_\Psi} g(\varphi(\overrightarrow{e}))}\Big(2\sqrt{\displaystyle \prod_{\overrightarrow{e} \in C_\Psi} g(\varphi(\overrightarrow{e}))} -[\varphi(C_\Psi) + f(\varphi(C_\Psi))]\Big)$ where the product runs over all component $1$-trees  $\Psi$ having unique cycle $C_\Psi$.
\end{thm}
\begin{proof}
	By suitable reordering of vertices and edges, if necessary, we can make the matrix $L_g(\Phi_f)$ as a block diagonal matrix where the blocks corresponds to the $1$-tree components of the $1$-forest. Then, by Lemma \ref{lem2}, determinant $\det(L_g(\Phi_f)) = \displaystyle \prod_{\Psi \in \mathfrak{E}(G)} \det(L_g(\Psi)).$ Now by applying Theorem \ref{unicyclic} we get \\ $\det L_g(\Phi_f) =\displaystyle \prod_{\Psi \in \mathfrak{E}(G)} \sqrt{ \displaystyle \prod_{\overrightarrow{e} \notin C_\Psi} g(\varphi(\overrightarrow{e}))}\Big(2\sqrt{\displaystyle \prod_{\overrightarrow{e} \in C_\Psi} g(\varphi(\overrightarrow{e}))} -[\varphi(C_\Psi) + f(\varphi(C_\Psi))]\Big).$
\end{proof}

Now we can prove the matrix - tree theorem for skew gain graphs.

\begin{lem}\label{lem3}
	Let  $\Phi_f =(G,F^\times,\varphi,f)$ be a skew gain graph on $n$ vertices and $\Psi $ be a spanning subgraph of $\Phi_f$ having exactly $n$ edges. Then $\det(L_g(\Psi )) \neq 0$ implies $\Psi $ is an essential spanning subgraph of $\Phi_f.$
\end{lem}
\begin{proof}
	Let $\Psi $ be a spanning subgraph of $\Phi_f$ having exactly $n$ edges and let  $\det(L_g(\Psi )) \neq 0$. We have to prove $\Psi $ is an essential spanning subgraph of $\Phi_f.$ That is, we have to prove that the components of $\Psi $ are $1$-trees. \\
	By suitable ordering of vertices and edges, we can make the matrix $L_g(\Psi )$ as a block diagonal matrix $\diag(A_i)$ where the blocks $A_i$ corresponds to the components of $\Psi$. Thus, $\det(L_g(\Psi)) = \displaystyle \prod_{A_i \in \Psi} \det(L_g(A_i)).$\\
	If $\Psi $ contains an isolated vertex, then the matrix $L_g(\Psi)$ has a zero row which implies $\det(L_g(\Psi)) = 0,$ a contradiction.\\
	If $A_i$ is a tree for some $i,$ then by Theorem \ref{tree} we get $\det(L_g(A_i)) = 0$ which implies $\det(L_g(\Psi)) =0,$ again a contradiction.\\
	Claim: If $A_k$ is a component of $\Psi$ then $A_k$ have same number of edges and vertices.\\
	Suppose $A_k$, for some $k,$ has $p$ vertices and $p+t$ edges where $t \geq 1.$ Then the $n-p$ vertices and $n-p-t$ edges not in $A_k$ forms either a tree or a disconnected graph having trees as components. Both cases leads to $\det(L_g(\Psi)) =0,$ a contradiction.
	Hence our claim.\\
	Now all the components of $\Psi$ have same number of edges and vertices implies the components of $\Psi$ are $1$-trees. Hence $\Psi $ is a spanning $1$-forest. That is  $\Psi $ is an essential spanning subgraph of $\Phi_f$.	
\end{proof}

\begin{thm}
	If $\Phi_f =(G,F^\times,\varphi,f)$ is a skew gain graph on $n$ vertices, then \\
	$\det(L_g(\Phi_f) = \displaystyle \sum_{\Psi \in \mathfrak{E}(G)}\displaystyle \prod_{\psi \in \Psi}   \sqrt{ \displaystyle \prod_{\overrightarrow{e} \notin C_\psi} g(\varphi(\overrightarrow{e}))}\Big(2\sqrt{\displaystyle \prod_{\overrightarrow{e} \in C_{\psi}} g(\varphi(\overrightarrow{e}))} -[\varphi(C_{\psi}) + f(\varphi(C_{\psi}))]\Big)$  where the summation runs over all essential spanning subgraphs $\Psi$ of $\Phi_f$ and $\psi \in \Psi$ denotes the component $1$-trees $\psi$ in the spanning $1$-forest $\Psi.$
\end{thm}
\begin{proof}
	Since $L_g(\Phi_f)=\mathrm{H}(\Phi_f)\mathrm{H}^{\#}(\Phi_f)$, by Binet-Cauchy theorem \cite{jgb} we get,\\
	$\det(L_g(\Phi_f)) =  \displaystyle \sum_J \det(\mathrm{H}(J))\det(\mathrm{H^{\#}}(J)) = \displaystyle \sum_J\det{L_g(J)}$ \\where $J$ is a spanning subgraph of $G$ with exactly $n$ edges. Then by Lemma \ref{lem3} we get
	$\det(L_g(\Phi_f)) =  \displaystyle \sum_{\Psi \in \mathfrak{E}(G)}\det{L_g({\Psi})}$ where the summation runs over all essential spanning subgraphs of $\Phi_f.$ Hence by Theorem \ref{forest} we get the \\
	\begin{equation*}
	\det(L_g(\Phi_f)) = \displaystyle \sum_{\Psi \in \mathfrak{E}(G)}\displaystyle \prod_{\psi \in \Psi}   \sqrt{ \displaystyle \prod_{\overrightarrow{e} \notin C_\psi} g(\varphi(\overrightarrow{e}))}\Big(2\sqrt{\displaystyle \prod_{\overrightarrow{e} \in C_{\psi}} g(\varphi(\overrightarrow{e}))} -[\varphi(C_{\psi}) + f(\varphi(C_{\psi}))]\Big).
	\end{equation*}
\end{proof}

\section*{Acknowledgement}
The first author would like to acknowledge her gratitude to Department of Science and Technology, Govt. of India for the financial support under INSPIRE Fellowship scheme Reg No: IF180462.

\end{document}